\documentclass[11pt]{amsart}
\usepackage{amsthm}
\usepackage{amsmath, mathrsfs}
\usepackage{amssymb}
\usepackage{hyperref}
\usepackage{enumerate}
\usepackage{pgfplots}
\pgfplotsset{compat=1.18}
\usepackage{float}

\usepackage{latexsym,array}
\usepackage{amsfonts}
\usepackage{shadow}
\usepackage{tikz, tikz-cd}
\tikzset{
	commutative diagrams/.cd,
	arrow style=tikz,
	diagrams={>=latex}}
\usetikzlibrary{matrix, arrows, positioning, shapes, shapes.multipart}
\usepackage{rotating}
\usepackage{tabularx}
\usepackage[english]{babel}
\usepackage{graphicx}
\usepackage[section]{placeins}

\tikzstyle{decision} = [diamond, draw, fill=gray!15, 
    text width=5em, text badly centered, node distance=3cm, inner sep=0pt]
\tikzstyle{block} = [rectangle, draw, fill=gray!15, 
    text width=10.5em, text centered, rounded corners, minimum height=4em]
    \tikzstyle{blueblock} = [rectangle, draw, fill=gray!10!blue!25, 
    text width=10.5em, text centered, rounded corners, minimum height=4em]
\tikzstyle{line} = [draw, -latex']
\tikzstyle{cloud} = [draw, ellipse,fill=gray!40, node distance=3cm, %text width=6em, text badly centered,
    minimum height=4em]
\newcommand{\U}{\mathcal{U}}

\newcommand{\C}{\mathbb{C}}

\newcommand{\Z}{\mathbb{Z}}
\newcommand{\T}{\mathbb{T}}

\newcommand{\Wind}{{\rm{Wind}}}
\newcommand{\Tr}{{\rm{Tr}}}

\newtheorem{Pa}{Paper}[section]
\newtheorem{Tm}[Pa]{{\bf Theorem}}
\newtheorem{mainTm}{Theorem}

\newtheorem{mainCy}[mainTm]{Corollary}
\newtheorem*{Wg}{Warning}
\newtheorem*{Qn}{Question}

\newtheorem{La}[Pa]{{\bf Lemma}}

\newtheorem{Cy}[Pa]{{\bf Corollary}}

\newtheorem{Pn}[Pa]{{\bf Proposition}}

\tikzset{node distance=2cm, auto}

\date{}

\title[$\Z^2$ is flexibly stable]{$\Z^2$ is flexibly stable in the operator norm}
\author{Alexander Lubotzky}
\address{Department of Mathematics\\ Weizmann Institute of Science\\Rehovot \\ Israel}
\email{alex.lubotzky@mail.huji.ac.il}

\author{Guy Salomon}
\address{School of Mathematical Sciences\\ Holon Institute of Technology \\ Holon \\ Israel}
\email{guysal@hit.ac.il}

\subjclass[2020]{46L05, 47A55, 20F69}
\begin{document}
\maketitle
\begin{abstract}
A cornerstone of stability theory is Voiculescu's 1983 counterexample: he constructed a sequence of pairs of unitary matrices whose commutators converge to zero in the operator norm, but whose distances from the set of commuting unitary pairs remain bounded away from zero. Namely, the group $\mathbb{Z}^2$ is not stable in the operator norm. We prove, somewhat surprisingly, that stability is restored after an asymptotically negligible enlargement of the dimension. That is, the group $\mathbb{Z}^2$ is flexibly stable in the operator norm. This provides the first example, in any context, of a flexibly stable group that is not stable.

Building on a construction of Eckhardt, who produced finitely generated amenable groups that are very-flexibly stable but not flexibly stable in the normalized Hilbert--Schmidt norm, we show that the same groups exhibit the analogous separation in the operator norm: they are very-flexibly stable but not flexibly stable.
\end{abstract}
%\tableofcontents

% %

\section{Introduction}\label{Sec:intro}

A classical problem in operator theory, originating in mathematical physics, asks whether almost commuting unitary matrices must be close to commuting unitary matrices. In the operator norm, this question was answered negatively by Voiculescu \cite{Voiculescu83}, who constructed a sequence of pairs of unitary matrices whose commutators converge to zero, but whose distance from the set of commuting unitary pairs remains bounded away from zero. Since then, Voiculescu's example has become one of the cornerstones of the theory of almost commuting operators and has inspired a large body of work \cite{ExelLoring89,GongLin98,EilLorPed99,GerholdShalit24,DorOnHallKachkovskiy25}.

The purpose of this paper is to show that Voiculescu's obstruction disappears once one allows an asymptotically negligible increase in dimension. More precisely, denote by $\U_{n}$ the group of $n \times n$ unitary matrices, let $u_n,v_n\in \U_{d_n}$ and assume that 
\[
\|u_nv_n-v_nu_n\|\longrightarrow 0.
\]
We prove that there exist integers
\[
k_n=d_n+o(d_n),
\]
isometries
\[
W_n: \mathbb{C}^{d_n}\to \mathbb{C}^{k_n},
\]
and commuting unitaries $U_n,V_n\in \U_{k_n}$ such that
\[
\|W_n^*U_nW_n-u_n\|\longrightarrow 0
\qquad\text{and}\qquad
\|W_n^*V_nW_n-v_n\|\longrightarrow 0.
\]
Equivalently, in the language of group stability, we prove the following.

\begin{mainTm}[Theorem \ref{Tm:Z2_is_FS}]\label{mainTmA}
The group $\mathbb{Z}^2$ is flexibly stable in the operator norm.
\end{mainTm}

Flexible stability was introduced by Becker and Lubotzky
\cite{BeckerLubotzky20} as a weakening of the usual notion of stability,
in which almost representations may be corrected after a negligible enlargement of the underlying space.
While this notion seems to have some theoretical value (cf. \cite{BowenBurton20,Dogon23,ChapmanLubotzky25I,ChapmanLubotzky25II}), it has been unknown whether this additional flexibility
genuinely enlarges the class of stable groups.

There are several reasons why a separation between these notions is not automatic. For amenable groups, Eckhardt and Shulman proved that flexible stability coincides with stability in the normalized Hilbert--Schmidt norm \cite[Proposition~6.2]{EckhardtShulman23}, and the analogous statement is known for permutation stability \cite[Lemma~3.2]{Ioana20}. Relatedly, Lazarovich--Levit--Minsky proved that surface groups are flexibly stable with respect to permutations, although their permutation stability remains open \cite{LazarovichLevitMinsky25}. To the best of our knowledge, prior to the present work, no example, in any context, was known in which flexible stability was proved to be strictly weaker than stability.

Since Voiculescu's example shows that $\mathbb{Z}^2$ is not stable in the
operator norm, our main theorem yields the first such separation.

\begin{mainCy}[Corollary \ref{Cy:FS_is_not_S}]
Flexible stability and stability are not equivalent in the operator norm.
Indeed, the group $\mathbb{Z}^2$ is flexibly stable but not stable.
\end{mainCy}

We now recall the relevant definitions. We begin with the general framework,
which allows us to compare unitary stability with permutation stability. Let
\[
\mathcal{G}=(G_n,\rho_n)_{n\in\mathbb{N}}
\]
be a sequence of groups $G_n$ equipped with bi-invariant metrics $\rho_n$, and let $\Gamma$
be a countable discrete group. A \emph{$\mathcal{G}$-challenge} for $\Gamma$
is a sequence of maps
\[
\varphi_n\colon \Gamma\to G_n
\]
such that
\[
\rho_n\bigl(\varphi_n(gh),\varphi_n(g)\varphi_n(h)\bigr)
\longrightarrow 0
\qquad
\text{for every }g,h\in\Gamma.
\]
The group $\Gamma$ is said to be \emph{$\mathcal{G}$-stable} if, for every such challenge, there exist homomorphisms
\[
\pi_n\colon \Gamma\to G_n
\]
such that
\[
\rho_n\bigl(\varphi_n(g),\pi_n(g)\bigr)
\longrightarrow 0
\qquad
\text{for every }g\in\Gamma.
\]

Taking $G_n={\rm{Sym}}(d_n)$, equipped with the normalized Hamming metric
\[
d_{\mathrm{Ham}}(\sigma,\tau)
=
\frac{1}{d_n}
\bigl|\{j\in\{1,\ldots,d_n\}:\sigma(j)\neq\tau(j)\}\bigr|,
\]
gives \emph{permutation stability}. Taking $G_n=\U_{d_n}$, equipped with a metric induced by some matrix norm, gives \emph{unitary stability}, a notion which depends strongly on the chosen matrix norm.

Flexible and very-flexible stability require a way of comparing representations of different dimensions. We therefore now specialize to unitary groups equipped with normalized unitarily invariant matrix norms, which is the setting used throughout the paper. Parallel notions can also be defined for permutations. Our main results concern the operator norm, although several comparisons with the normalized Hilbert--Schmidt norm will be useful.

For $n\in\mathbb{N}$, let $M_n:=M_n(\C)$ be the C*-algebra of $n \times n$ complex matrices and let
\[
\bigl(\|\cdot\|_{*,n}\bigr)_{n\in\mathbb{N}}
\]
be a family of normalized unitarily invariant matrix norms, so that
\[
\|1_n\|_{*,n}=1
\qquad\text{and}\qquad
\|uav\|_{*,n}=\|a\|_{*,n}
\]
for every $a\in M_n$ and every pair of unitaries $u,v\in \U_n$. We usually suppress the
dimension from the notation and write simply $\|\cdot\|_*$. Throughout the
paper, unless otherwise stated, $\|\cdot\|$ denotes the operator norm. We will
also refer to the normalized Hilbert--Schmidt norm, defined by
\[
\|a\|_2
=
\left(\frac{1}{n}\operatorname{Tr}(a^*a)\right)^{1/2},
\qquad a\in M_n.
\]

In this setting, a challenge, also known as an \emph{asymptotic representation} of $\Gamma$ with respect to
$\|\cdot\|_*$, is just a sequence of maps
\[
\varphi_n\colon \Gamma\to \U_{d_n}
\]
such that
\[
\|\varphi_n(gh)-\varphi_n(g)\varphi_n(h)\|_*
\longrightarrow 0
\qquad
\text{for every }g,h\in\Gamma.
\]

We say that $\Gamma$ is \emph{stable}, \emph{flexibly stable}, or
\emph{very-flexibly stable} with respect to $\|\cdot\|_*$ if, for every
asymptotic representation $(\varphi_n)$, there exist integers $k_n\geq d_n$,
isometries
$
W_n\colon \mathbb{C}^{d_n}\to \mathbb{C}^{k_n},
$
and genuine unitary representations
\[
\pi_n\colon \Gamma\to \U_{k_n}
\]
such that
\[
\|W_n^*\pi_n(g)W_n-\varphi_n(g)\|_*
\longrightarrow 0
\qquad
\text{for every }g\in\Gamma,
\]
subject to the following dimension constraints:
\[
\begin{array}{c|c}
\text{notion} & \text{dimension constraint}\\
\hline
\text{stability} & k_n=d_n,\\[1mm]
\text{flexible stability} & k_n/d_n\longrightarrow 1,\\[1mm]
\text{very-flexible stability} & \text{no restriction on }k_n/d_n.
\end{array}
\]
Thus, stability requires a correction in the original dimension, flexible
stability permits only an asymptotically negligible enlargement, and
very-flexible stability permits an arbitrary finite-dimensional enlargement.

The same definitions extend verbatim to unital C*-algebras by replacing
group representations with unital *-representations. Equivalently, a group
$\Gamma$ is stable (respectively, flexibly stable or very-flexibly stable) if
and only if its full group C*-algebra $C^*(\Gamma)$ enjoys the corresponding
property.

A broad source of very-flexibly stable C*-algebras is provided by residual
finite-dimensionality together with lifting properties. Eckhardt and Shulman
proved that every residually finite-dimensional C*-algebra with the lifting
property is very-flexibly stable in the normalized Hilbert--Schmidt norm
\cite[Theorem~6.4]{EckhardtShulman23}. More recently, Fournier-Facio and
Willett extended this result to arbitrary normalized unitarily invariant
matrix norms: every separable unital C*-algebra with the local lifting
property and residual finite-dimensionality is very-flexibly stable
\cite{FournierFacioWillett26}.

Consequently, every nuclear residually finite-dimensional unital
C*-algebra is very-flexibly stable. In particular, if $\Gamma$ is amenable
and residually finite, then $C^*(\Gamma)$ is nuclear and residually
finite-dimensional, and hence $\Gamma$ is very-flexibly stable with respect
to every normalized unitarily invariant matrix norm.

Since $C^*(\mathbb{Z}^2)\cong C(\mathbb{T}^2)$
is nuclear and residually finite-dimensional, the general criterion already
shows that $\mathbb{Z}^2$ is very-flexibly stable with respect to every
normalized unitarily invariant matrix norm. The content of our first main
theorem is that, in the operator norm, the arbitrary enlargement permitted by
very-flexible stability can always be replaced by an asymptotically negligible
one.

Our second main result shows that this improvement is not available for every
very-flexibly stable group. Building on a construction of Eckhardt
\cite{Eckhardt25}, we show the following.

\begin{mainTm}[Corollary \ref{Cy:VFS_is_not_FS}]\label{mainTmC}
There exist finitely generated groups that are very-flexibly stable but not
flexibly stable, in the operator norm.
\end{mainTm}

Together, our two main results show that both implications
\[
\text{stability}
\Longrightarrow
\text{flexible stability}
\Longrightarrow
\text{very-flexible stability}
\]
are strict in the operator norm.

Theorem \ref{mainTmC} will be proved in Section~\ref{Sec:VFSnonFS}, while Theorem \ref{mainTmA} will be proved in Section~\ref{Sec:Z2isFS}.
The main ingredient in its proof is a result proved independently by Gong--Lin \cite{GongLin98} and Eilers--Loring--Pedersen \cite{EilLorPed99}: almost commuting unitary matrices with vanishing winding-number obstruction can be approximated by commuting
unitary matrices. Our contribution is to show that this obstruction can always
be eliminated by adding only $o(d_n)$ dimensions.

%%%%%%%%%%%%

\section{$\mathbb Z^2$ is flexibly stable}\label{Sec:Z2isFS}

In this section, we show that $\mathbb Z^2$, or, equivalently, its full group C*-algebra $C^*(\mathbb Z^2)\cong C(\T^2)$, is flexibly stable.

\subsection{The invariant $w(u,v)$}
For $u,v \in \U_n$, consider the line-segment path from $uv$ to $vu$,
\[
[0,1] \to M_n,
\qquad
t \mapsto (1-t)uv+tvu .
\]
If each matrix along this path is invertible, then its determinant defines a closed curve in $\C\setminus\{0\}$, since the endpoints have the same determinant:
\[
\det(uv)=\det(vu).
\]

In this paper, we will only consider this path in the case where $u$ and $v$ almost commute. In particular, we may assume that $\|[u,v]\|<2$.
Thus,
\[
\|vuv^*u^*-1\|=
\|vu-uv\|=
\|[u,v]\|<2.
\]
Since $vuv^*u^*$ is unitary, this implies that $-1$ is not in its spectrum, namely $-1 \notin \sigma(vuv^*u^*)$.

Now, for every $\lambda\in\sigma(vuv^*u^*)\subseteq\T$, the line segment from $1$ to $\lambda$ avoids $0$, since $\lambda\neq -1$. Hence,
\[
(1-t)1+tvuv^*u^*
\]
is invertible for all $t\in[0,1]$, and therefore so is $(1-t)uv+tvu=((1-t)1+tvuv^*u^*)uv$. In particular, the closed curve
\[
t\mapsto \det((1-t)uv+tvu)
\]
does not vanish.

Now, recall that if $\gamma:[0,1]\to \C\setminus\{0\}$ is a closed path, then its {\em winding number} around the origin is given by
\begin{equation} \label{Eq:winding}
    \Wind(\gamma,0)
   :=\frac{1}{2\pi i}\int_\gamma \frac{dz}{z}
   =\frac{1}{2\pi i}\int_0^1 \frac{\gamma'(t)}{\gamma(t)}\,dt .
\end{equation}
In particular, $\Wind(\gamma,0)$ is invariant under homotopy in $\C\setminus\{0\}$.

The {\em winding number invariant} of two unitaries $u,v \in \U_n$ is given by
\[
w(u,v):=\Wind\big(\det((1-t)uv+tvu),0\big).
\]
%%%%

We will use the following elementary additivity property. If $u_j,v_j\in \U_{n_j}$, $j=1,2$, and the corresponding winding numbers are defined, then
\[
w(u_1\oplus u_2,v_1\oplus v_2)
=
w(u_1,v_1)+w(u_2,v_2).
\]
Indeed, the determinant of the direct-sum path is the product of the two determinant paths, and winding number is additive under products of closed curves in $\C\setminus\{0\}$, as can be seen from the second equality of \eqref{Eq:winding}.

%%%
The invariant $w(\cdot,\cdot)$ also admits the following useful logarithmic description. Let
\[
\log:\C\setminus(-\infty,0]\to\C
\]
denote the principal branch of the logarithm. If $\|[u,v]\|<2$, then
$-1$ is not in the spectrum of the normal matrix $vuv^*u^*$, and therefore $\log(vuv^*u^*)$ is well-defined by functional calculus. Moreover, the path
\[
[0,1]\to M_n,
\qquad
t\mapsto \exp\big(t\log(vuv^*u^*)\big)uv
\]
is homotopic in ${\rm GL}_n(\mathbb C)$ to the line-segment path
\[
t\mapsto (1-t)uv+tvu.
\]
Thus, their determinants are homotopic as curves in $\C\setminus\{0\}$, so that the two curves have the same winding number. See \cite[Lemma 3.1]{Exel93} for further details. 

Let
\[
\gamma(t):=\det\left(\exp\big(t\log(vuv^*u^*)\big)uv\right), \quad t\in[0,1].
\]
Since $\det(\exp(a))=\exp(\Tr(a))$ for every $a\in M_n$, we have
\[
\gamma(t)
=
\exp\left(t\Tr(\log(vuv^*u^*))\right)\det(u)\det(v).
\]
Therefore
\[
w(u,v)
  =\frac{1}{2\pi i}\Tr(\log(vuv^*u^*)).
\]
%See \cite[Lemma 3.1]{Exel93} for further details.
%%%

\begin{La}\label{La:Wind_over_dim_to_zero}
Let $u_n,v_n \in \U_{d_n}$ be unitaries such that
\[
\lim_{n \to \infty}\|[u_n,v_n]\|=0.
\]
Then $w(u_n,v_n)$ is defined for all sufficiently large $n$, and
\[
\lim_{n \to \infty} \frac{w(u_n,v_n)}{d_n}=0.
\]
\end{La}

\begin{proof}
Since $\log$ is holomorphic on $\C\setminus(-\infty,0]$, it is in particular Lipschitz on any compact set in $\C\setminus(-\infty,0]$. Thus, there exists $C>0$ such that every $z \in \C$ with $|z-1| \leq \frac{1}{2}$ satisfies
\[
|\log(z)|\leq C |z-1|.
\]
As
\[
\|v_nu_nv_n^*u_n^* -1 \|=\|v_nu_n-u_nv_n\|=\left\|[u_n,v_n]\right\| \longrightarrow 0,
\]
for all sufficiently large $n$, we have
\[
\|\log(v_nu_nv_n^*u_n^*)\|\leq C \|v_nu_nv_n^*u_n^*-1\|= C \left\|[u_n,v_n]\right\|.
\]
Thus,
\[
\begin{split}
    |w(u_n,v_n)|&=\frac{1}{2\pi }\left|\Tr(\log(v_nu_nv_n^*u_n^*))\right| \\
    &\leq \frac{1}{2\pi } d_n \left\|\log(v_nu_nv_n^*u_n^*)\right\| \\
    &\leq \frac{C}{2\pi } d_n  \left\|[u_n,v_n]\right\|.
\end{split}
\]
Dividing by $d_n$ gives the required result.
\end{proof}

\subsection{Voiculescu's matrices}

Fix $n \geq 2$, and set $\zeta_n=e^{2\pi i/n}$. Let
\[
\Omega_n=\mathrm{diag}(1,\zeta_n,\ldots,\zeta_n^{n-1}),
\]
and let $S_n$ denote the $n\times n$ cyclic shift matrix. Then $S_n$ and $\Omega_n$ are unitaries satisfying
\[
S_n\Omega_n=\zeta_n\Omega_n S_n.
\]
Consequently,
\[
[S_n,\Omega_n]
   =S_n\Omega_n-\Omega_n S_n
   =(\zeta_n-1)\Omega_n S_n,
\]
and hence
\[
\|[S_n,\Omega_n]\|=|\zeta_n-1|\longrightarrow 0.
\]
Thus the pairs $(S_n,\Omega_n)$ form a sequence of almost commuting unitary matrices.

On the other hand, their winding-number obstruction is nontrivial:
\[
w(S_n,\Omega_n)=-1,
\qquad
w(\Omega_n,S_n)=1.
\]
This is the obstruction exploited by Voiculescu in his proof that $\mathbb Z^2$ is not stable \cite{Voiculescu83}. See also Exel--Loring \cite{ExelLoring89}, where the winding-number invariant is used explicitly.

More than a decade later, Gong--Lin \cite{GongLin98} and, independently, Eilers--Loring--Pedersen \cite{EilLorPed99} showed that this winding-number obstruction is the only obstruction.

\begin{Tm}[Gong--Lin \cite{GongLin98}, Eilers--Loring--Pedersen \cite{EilLorPed99}]\label{Tm:Pert_of_unitaries}
    If $u_n, v_n \in \U_{d_n}$ with $\lim_{n \to \infty}\|[u_n,v_n]\|=0$ and $w(u_n,v_n)=0$ for all sufficiently large $n$, then there are commuting unitaries $U_n,V_n \in \U_{d_n}$ such that
    \[
    \lim_{n\to \infty}\|U_n - u_n \| =0 \quad \text{and} \quad  \lim_{n\to \infty}\|V_n - v_n \| =0. 
    \]
\end{Tm}

%%%%%%%%%%

\subsection{The main theorem}
\begin{Tm}[$\mathbb Z^2$ is flexibly stable]\label{Tm:Z2_is_FS}
    Let $d_n$ be a sequence of positive integers, and let $u_n,v_n \in \U_{d_n}$ be unitaries with
    \[
    \lim_{n\to \infty}\|[u_n,v_n]\|=0.
    \]
    Then there are $k_n \in \mathbb N$ and isometries $W_n:\C^{d_n}\to \C^{k_n}$ such that
    \[
    \lim_{n\to \infty} \frac{k_n}{d_n}=1,
    \]
    and commuting unitaries $U_n,V_n \in \U_{k_n}$ such that
    \[
       \lim_{n \to \infty} \|W_n^*U_nW_n - u_n\|=0 \quad \text{and} \quad 
         \lim_{n \to \infty} \|W_n^*V_nW_n - v_n\|=0. 
    \]
\end{Tm}
%%%
\begin{proof}
    Set $w_n:=w(u_n,v_n)$ for all sufficiently large $n$, and define $w_n:=0$ for the remaining finitely many $n$. By Lemma \ref{La:Wind_over_dim_to_zero},
    \[
    \lim_{n \to \infty}\frac{w_n}{d_n}= 0.
    \]
    Choose a sequence of positive integers $m_n\to\infty$ such that
    \[
     \lim_{n \to \infty}\frac{w_n m_n}{d_n}= 0.
    \]
    Set
    \[
    k_n:=d_n+|w_n|m_n.
    \]

    If $w_n\geq 0$, define
    \[
    X_n:=u_n\oplus S_{m_n}^{\oplus |w_n|},
    \qquad
    Y_n:=v_n\oplus \Omega_{m_n}^{\oplus |w_n|}.
    \]
    Otherwise, if $w_n<0$, define
    \[
    X_n:=u_n\oplus \Omega_{m_n}^{\oplus |w_n|},
    \qquad
    Y_n:=v_n\oplus S_{m_n}^{\oplus |w_n|}.
    \]
    Then $X_n,Y_n\in \U_{k_n}$.

    By additivity of the winding number under direct sums, and since
    \[
    w(S_{m_n},\Omega_{m_n})=-1,
    \qquad
    w(\Omega_{m_n},S_{m_n})=1,
    \]
    we have
    \[
    w(X_n,Y_n)=0
    \]
    for all sufficiently large $n$.

    Moreover,
    \[
    \begin{split}
    \|[X_n,Y_n]\|
    &\leq
    \max\{\|[u_n,v_n]\|,\|[S_{m_n},\Omega_{m_n}]\|\}\\
     &=
    \max\{\|[u_n,v_n]\|,|\zeta_{m_n}-1|\}\longrightarrow 0.
   \end{split}
    \]
    Therefore, by Theorem \ref{Tm:Pert_of_unitaries}, there are commuting unitaries
    $U_n,V_n\in \U_{k_n}$ such that
    \[
     \lim_{n \to \infty}\|U_n-X_n\|=0,
    \qquad
     \lim_{n \to \infty}\|V_n-Y_n\|=0.
    \]

    Let $W_n:\C^{d_n}\to\C^{k_n}$ be the canonical isometric embedding onto the first $d_n$ coordinates. Then
    \[
     \lim_{n \to \infty}\|W_n^*U_nW_n-u_n\|=0,
    \qquad
     \lim_{n \to \infty}\|W_n^*V_nW_n-v_n\|=0.
    \]
    Finally,
    \[
    \frac{k_n}{d_n}
    =
    1+\frac{|w_n|m_n}{d_n}\to 1.
    \]
\end{proof}

For amenable groups, stability in the normalized Hilbert–Schmidt norm admits a trace-theoretic characterization due to Shulman and Hadwin \cite{ShulmanHadwin18b}. Using this characterization, Eckhardt and Shulman showed that, in this setting, allowing flexible enlargements of the matrix size does not change the notion of stability.

\begin{Tm}[Eckhardt and Shulman, Proposition 6.2 of \cite{EckhardtShulman23}]
Let $\Gamma$ be amenable. Then $\Gamma$ is flexibly stable in the normalized Hilbert–Schmidt norm if and only if it is stable in the normalized Hilbert–Schmidt norm.
\end{Tm}

The operator norm behaves differently. Indeed, our result on two almost commuting unitaries shows that the distinction between stability and flexible stability is already visible for the simplest non-cyclic abelian group.

\begin{Cy}\label{Cy:FS_is_not_S}
The group $\mathbb Z^2$ is flexibly stable, but not stable, in the operator norm. Consequently, flexible stability and stability are not equivalent in the operator norm, even for abelian groups.
\end{Cy}

\subsection{Remarks} We end this section with a warning about terminology and with a natural question.
\begin{Wg}\em
There are several inequivalent ways to ``stabilize'' matricial stability, and the terminology in the literature is not completely uniform. In particular, Dadarlat's {\em weak matricial stability} \cite{Dadarlat21} is a direct-sum stabilization property: roughly speaking, one asks whether an asymptotic representation $\varphi_n$ can be made close to a genuine representation after {\em adding another genuine finite-dimensional representation}. Thus, the comparison is of the form
\[
\varphi_n \oplus \sigma_n \approx \pi_n,
\]
where $\sigma_n$ and $\pi_n$ are genuine representations. In some related formulations, the added summand is required to be trivial or is viewed as an auxiliary direct-sum correction. This is different from the notion of flexible stability used in this paper, where the approximation is by compression from a slightly larger representation, rather than by adding a direct summand to the almost-representation.

This distinction is important. Dadarlat's cohomological obstructions apply to ordinary matricial stability and, under suitable hypotheses, also to certain direct-sum stabilized variants. They should not be interpreted as obstructions to flexible stability in the sense used here. Indeed, $\mathbb Z^2$ has nontrivial second homology, and hence is obstructed from being stable in the operator norm; nevertheless, the main result of this paper shows that $\mathbb Z^2$ is flexibly stable in the operator norm.
\end{Wg}

\begin{Qn}\em
Is $\mathbb Z^3$ flexibly stable in the operator norm, or is it only very-flexibly stable? If it is not flexibly stable, is the required dimension increase nevertheless linear in the original dimension?

More precisely, let $A$ be a unital $C^*$-algebra. Given a unital $*$-homomorphism
\[
\varphi:A\to \prod_n M_{d_n}\big/ \bigoplus_n M_{d_n},
\]
define $\Delta(\varphi)$ to be the infimum of
\[
\limsup_{n\to\infty}\frac{k_n}{d_n}
\]
over all choices of integers $k_n\geq d_n$, isometries
$W_n:\mathbb C^{d_n}\to \mathbb C^{k_n}$, and unital $*$-homomorphisms
\[
\widetilde\varphi:A\to \prod_n M_{k_n}
\]
such that 
\[
\varphi(a)
=
\pi\left(
\bigl(W_n^*\widetilde\varphi_n(a)W_n\bigr)_{n\in\mathbb N}
\right)
\qquad
\text{for every } a\in A,
\]
where $\pi$ denotes the quotient map $\prod_n M_{d_n}\to \prod_n M_{d_n}\big/ \bigoplus_n M_{d_n}$.
(If no such triple $(\widetilde\varphi,W_n,k_n)$ exists, we set
$\Delta(\varphi)=+\infty$.) Finally, define
\[
D(A)
:=
\sup_{\varphi}
\Delta(\varphi),
\]
where the supremum is taken over all unital $*$-homomorphisms
\[
\varphi:A\to \prod_n M_{d_n}\big/ \bigoplus_n M_{d_n}
\]
and over all sequences $(d_n)$.

With this terminology, flexible stability of a unital C*-algebra $A$ is equivalent % by a standard diagonal argument, 
to the assertion that $D(A)=1$. In particular, the main result of this paper says that
$D(C^*(\mathbb Z^2))=D(C(\mathbb T^2))=1$. What is $D(C^*(\mathbb Z^3))$?

More generally, can one estimate $D(C^*(\Gamma))$ for amenable groups $\Gamma$, or at least for abelian groups?
\end{Qn}

%%%%%%%%%%%%%%%%%
\section{Very-flexibly stable groups that are not flexibly stable}
\label{Sec:VFSnonFS}

The following theorem is the operator-norm counterpart of a result appearing in Eckhardt's work for the normalized Hilbert--Schmidt norm. In fact, the result can be extracted from the proof of \cite[Theorem~2.14]{Eckhardt25}, together with the discussion of flexible Hilbert--Schmidt stability in \cite[Section~6]{EckhardtShulman23}. For completeness, we give the short argument.

Throughout this section, $\tau_n$ denotes the normalized trace on $M_n$. A \emph{finite-dimensional trace} on a group $\Gamma$ is a function of the form
\[
g\mapsto \tau_n(\pi(g)),
\]
where $\pi\colon\Gamma\to\U_n$ is a finite-dimensional unitary representation. Convergence of traces is understood pointwise on $\Gamma$.

\begin{Tm}\label{Tm:VFS_not_FS}
Let $\Gamma$ be an amenable group, and suppose that it has a normal subgroup
$\Lambda\triangleleft\Gamma$ such that the trace
\[
\tau(g)=
\begin{cases}
1,&g\in\Lambda,\\
0,&g\notin\Lambda,
\end{cases}
\]
is not a pointwise limit of finite-dimensional traces. Then $\Gamma$ is not
flexibly stable in the operator norm.
\end{Tm}

We first record an elementary estimate.

\begin{Pn}\label{Pn:FS=>Traces_to_0}
Let $k_n\geq d_n$, let
$W_n\colon\C^{d_n}\to\C^{k_n}$ be isometries, and suppose that
\[
\lim_{n \to \infty} \frac{k_n}{d_n} = 1.
\]
Then, for every sequence $u_n\in\U_{k_n}$,
\[
\lim_{n \to \infty} \left( \tau_{d_n}(W_n^*u_nW_n)-\tau_{k_n}(u_n) \right) = 0.
\]
\end{Pn}

\begin{proof}
    Set $P_n:=W_nW_n^*$. Then
    \[
    \tau_{k_n}(u_n)
    =\frac{1}{k_n}\Tr(u_n)
    =\frac{1}{k_n}\Tr(P_nu_n)+\frac{1}{k_n}\Tr((1-P_n)u_n),
    \]
    and also
    \[
    \tau_{d_n}(W_n^*u_nW_n)=  \frac{1}{d_n}\Tr(P_n u_n).
    \]
    Therefore,
    \[
    \begin{split}
        \tau_{d_n}(W_n^*u_n W_n) - \tau_{k_n}(u_n)
        &=\left(\frac{1}{d_n}-\frac{1}{k_n}\right)\Tr(P_n u_n) - \frac{1}{k_n}\Tr((1-P_n)u_n).
    \end{split}
    \]
    Since $P_n$ is an orthogonal projection of rank $d_n$, if $(e_i)_{i=1}^{d_n}$ is an orthonormal basis for its range, then
\[
\left|\Tr(P_nu_n)\right|
=
\left|\sum_{i=1}^{d_n}\langle u_ne_i,e_i\rangle\right|
\leq
\sum_{i=1}^{d_n}\|u_ne_i\|\,\|e_i\|
=
d_n,
\]
because $u_n$ is unitary. Similarly, since $1-P_n$ has rank $k_n-d_n$,
\[
\left|\Tr((1-P_n)u_n)\right|\leq k_n-d_n.
\]
    % Since $u_n$ is unitary
    % \[
    % |\Tr(P_n u_n)| \leq d_n, \qquad |\Tr((1-P_n)u_n)| \leq k_n - d_n.
    % \]
    Thus
\[
    \begin{split}
       \left| \tau_{d_n}(W_n^*u_n W_n) - \tau_{k_n}(u_n) \right|
        & \leq \left(\frac{1}{d_n}-\frac{1}{k_n}\right)d_n + \frac{1}{k_n}(k_n-d_n)\\
        &=2\left(1-\frac{d_n}{k_n} \right) \longrightarrow 0.
    \end{split}
    \]
\end{proof}

\begin{La}\label{La:FS=>TraceStable}
Suppose that $\Gamma$ is flexibly stable in the operator norm and that
$\varphi_n\colon\Gamma\to\U_{d_n}$ satisfies
\[
\lim_{n \to \infty} \|\varphi_n(gh)-\varphi_n(g)\varphi_n(h)\| = 0
\qquad \text{for all } g,h\in\Gamma.
\]
Then there exist integers $k_n\geq d_n$ with $\lim_{n \to \infty} k_n/d_n = 1$ and
representations $\pi_n\colon\Gamma\to\U_{k_n}$ such that
\[
\lim_{n \to \infty} \left( \tau_{d_n}(\varphi_n(g))-\tau_{k_n}(\pi_n(g)) \right) = 0
\qquad \text{for all } g\in\Gamma.
\]
\end{La}

\begin{proof}
By flexible stability, there are isometries
$W_n\colon\C^{d_n}\to\C^{k_n}$ and representations
$\pi_n\colon\Gamma\to\U_{k_n}$ such that
\[
\lim_{n \to \infty} \|W_n^*\pi_n(g)W_n-\varphi_n(g)\| = 0
\qquad \text{for all } g\in\Gamma.
\]
Consequently, for any $g\in \Gamma$,
    \[
    \begin{split}
     |\tau_{k_n}&(\pi_n(g))-\tau_{d_n}(\varphi_n(g))|\\
    &\leq \left |\tau_{k_n}(\pi_n(g))-\tau_{d_n}(W_n^*\pi_n(g)W_n)\right|+\left |\tau_{d_n}(W_n^*\pi_n(g)W_n)-\tau_{d_n}(\varphi_n(g))\right|\\
     &\leq \left |\tau_{k_n}(\pi_n(g))-\tau_{d_n}(W_n^*\pi_n(g)W_n)\right|+\|W_n^*\pi_n(g)W_n-\varphi_n(g)\|\longrightarrow 0,
    \end{split}
    \]
    where the first summand tends to zero by Proposition \ref{Pn:FS=>Traces_to_0}.
\end{proof}

\begin{proof}[Proof of Theorem~\ref{Tm:VFS_not_FS}]
The proof of \cite[Theorem~2.14]{Eckhardt25} provides maps
$\varphi_n\colon\Gamma\to\U_{d_n}$ such that
\[
\lim_{n \to \infty} \|\varphi_n(gh)-\varphi_n(g)\varphi_n(h)\| = 0
\qquad \text{for all } g,h\in\Gamma,
\]
and
\[
\lim_{n \to \infty} \tau_{d_n}(\varphi_n(g)) = \tau(g)
\qquad \text{for all } g\in\Gamma.
\]
Assume, toward a contradiction, that $\Gamma$ is flexibly stable in the operator norm.
Then, by Lemma \ref{La:FS=>TraceStable}, there are $k_n \geq d_n$ and finite-dimensional representations $\pi_n: \Gamma \to \U_{k_n}$ such that
\[
    \lim_{n \to \infty} \left( \tau_{d_n}(\varphi_n(g)) - \tau_{k_n}(\pi_n(g))\right) =0 \quad \text{for all } g\in \Gamma.
\] 
Thus,
\[
\lim_{n \to \infty} \tau_{k_n}\left(\pi_n(g) \right) = \tau(g) \quad \text{for all $g \in \Gamma$},
\]
which is a contradiction.
\end{proof}

For every prime $p$, let $\widetilde G_p$ be the subgroup of upper triangular matrices in
$GL_5(\mathbb Z[\tfrac{1}{p}])$ given by
\[
\widetilde G_p
:=
\left\{
\begin{bmatrix}
1&\ast&\ast&\ast&\ast\\
0&p^n&\ast&\ast&\ast\\
0&0&p^k&\ast&\ast\\
0&0&0&1&m\\
0&0&0&0&1
\end{bmatrix}
:
n,k,m\in\mathbb Z
\right\},
\]
where each $\ast$ denotes an arbitrary element of
$\mathbb Z[\tfrac{1}{p}]$. Eckhardt shows that $\widetilde G_p$ is finitely
presented, residually finite, and solvable, and that it admits a normal
subgroup satisfying the assumptions of
Theorem~\ref{Tm:VFS_not_FS}.

\begin{Cy}\label{Cy:VFS_is_not_FS}
For every prime $p$, the group $\widetilde G_p \leq GL_5(\mathbb Z[\tfrac{1}{p}])$ is very-flexibly stable but
not flexibly stable in the operator norm.
\end{Cy}

\begin{proof}
By Theorem~\ref{Tm:VFS_not_FS}, the group $\widetilde G_p$ is not flexibly stable in the operator norm. On the other hand, since $\widetilde G_p$ is finitely generated, residually finite, and amenable, its full group C*-algebra is residually finite-dimensional and nuclear, and therefore has the lifting property. It follows from \cite{FournierFacioWillett26} that $\widetilde G_p$ is very-flexibly stable in the operator norm.
\end{proof}

%%%%%%%%%%%%%%%%%%%%%%%%%%%%%%%%
\bibliographystyle{abbrv}
\bibliography{refs}  

@article {Dadarlat21,
    AUTHOR = {Dadarlat, Marius},
     TITLE = {Obstructions to matricial stability of discrete groups and
              almost flat {K}-theory},
   JOURNAL = {Adv. Math.},
  FJOURNAL = {Advances in Mathematics},
    VOLUME = {384},
      YEAR = {2021},
     PAGES = {Paper No. 107722, 29},
      ISSN = {0001-8708,1090-2082},
   MRCLASS = {20F65 (19J05 19K35 20E26 46L80)},
  MRNUMBER = {4238917},
       DOI = {10.1016/j.aim.2021.107722},
       URL = {https://doi-org.ezproxy.weizmann.ac.il/10.1016/j.aim.2021.107722},
}

@article {BowenBurton20,
    AUTHOR = {Bowen, Lewis and Burton, Peter},
     TITLE = {Flexible stability and nonsoficity},
   JOURNAL = {Trans. Amer. Math. Soc.},
  FJOURNAL = {Transactions of the American Mathematical Society},
    VOLUME = {373},
      YEAR = {2020},
    NUMBER = {6},
     PAGES = {4469--4481},
      ISSN = {0002-9947,1088-6850},
   MRCLASS = {20F69 (20G20 20G40)},
  MRNUMBER = {4105530},
MRREVIEWER = {Tushar\ Das},
       DOI = {10.1090/tran/8047},
       URL = {https://doi-org.ezproxy.weizmann.ac.il/10.1090/tran/8047},
}

@article {ChapmanLubotzky25I,
    AUTHOR = {Chapman, Michael and Lubotzky, Alexander},
     TITLE = {Stability of homomorphisms, coverings and cocycles {I}:
              equivalence},
   JOURNAL = {Eur. J. Math.},
  FJOURNAL = {European Journal of Mathematics},
    VOLUME = {11},
      YEAR = {2025},
    NUMBER = {1},
     PAGES = {Paper No. 13, 38},
      ISSN = {2199-675X,2199-6768},
   MRCLASS = {20F65},
  MRNUMBER = {4859639},
MRREVIEWER = {Peter\ Ha\"{\i}ssinsky},
       DOI = {10.1007/s40879-024-00803-8},
       URL = {https://doi-org.ezproxy.weizmann.ac.il/10.1007/s40879-024-00803-8},
}

@article {ChapmanLubotzky25II,
    AUTHOR = {Chapman, Michael and Lubotzky, Alexander},
     TITLE = {Stability of homomorphisms, coverings and cocycles {II}:
              examples, applications and open problems},
   JOURNAL = {Adv. Math.},
  FJOURNAL = {Advances in Mathematics},
    VOLUME = {463},
      YEAR = {2025},
     PAGES = {Paper No. 110117, 38},
      ISSN = {0001-8708,1090-2082},
   MRCLASS = {20F65 (05E45 20B99 20E99)},
  MRNUMBER = {4860041},
MRREVIEWER = {Peter\ Ha\"{\i}ssinsky},
       DOI = {10.1016/j.aim.2025.110117},
       URL = {https://doi-org.ezproxy.weizmann.ac.il/10.1016/j.aim.2025.110117},
}

@article {Ioana20,
    AUTHOR = {Ioana, Adrian},
     TITLE = {Stability for product groups and property ({$\tau$})},
   JOURNAL = {J. Funct. Anal.},
  FJOURNAL = {Journal of Functional Analysis},
    VOLUME = {279},
      YEAR = {2020},
    NUMBER = {9},
     PAGES = {108729, 32},
      ISSN = {0022-1236,1096-0783},
   MRCLASS = {20B99 (20E22 20F69 22A05)},
  MRNUMBER = {4134896},
MRREVIEWER = {Alain\ Valette},
       DOI = {10.1016/j.jfa.2020.108729},
       URL = {https://doi-org.ezproxy.weizmann.ac.il/10.1016/j.jfa.2020.108729},
}

@article {Dogon23,
    AUTHOR = {Dogon, Alon},
     TITLE = {Flexible {H}ilbert-{S}chmidt stability versus hyperlinearity
              for property ({T}) groups},
   JOURNAL = {Math. Z.},
  FJOURNAL = {Mathematische Zeitschrift},
    VOLUME = {305},
      YEAR = {2023},
    NUMBER = {4},
     PAGES = {Paper No. 58, 20},
      ISSN = {0025-5874,1432-1823},
   MRCLASS = {20P05 (20E26 20J06 22D10 22D25 22D55)},
  MRNUMBER = {4660281},
MRREVIEWER = {Paul\ Jolissaint},
       DOI = {10.1007/s00209-023-03387-3},
       URL = {https://doi-org.ezproxy.weizmann.ac.il/10.1007/s00209-023-03387-3},
}

@article {LazarovichLevitMinsky25,
    AUTHOR = {Lazarovich, Nir and Levit, Arie and Minsky, Yair},
     TITLE = {Surface groups are flexibly stable},
   JOURNAL = {J. Eur. Math. Soc. (JEMS)},
  FJOURNAL = {Journal of the European Mathematical Society (JEMS)},
    VOLUME = {27},
      YEAR = {2025},
    NUMBER = {4},
     PAGES = {1739--1768},
      ISSN = {1435-9855,1435-9863},
   MRCLASS = {20B07 (15A24 20F67 20F69 57K20)},
  MRNUMBER = {4875616},
MRREVIEWER = {V.\ V.\ Benyash-Krivets},
       DOI = {10.4171/jems/1406},
       URL = {https://doi-org.ezproxy.weizmann.ac.il/10.4171/jems/1406},
}

@article {BeckerLubotzky20,
    AUTHOR = {Becker, Oren and Lubotzky, Alexander},
     TITLE = {Group stability and {P}roperty ({T})},
   JOURNAL = {J. Funct. Anal.},
  FJOURNAL = {Journal of Functional Analysis},
    VOLUME = {278},
      YEAR = {2020},
    NUMBER = {1},
     PAGES = {108298, 20},
      ISSN = {0022-1236,1096-0783},
   MRCLASS = {20F69 (22E99)},
  MRNUMBER = {4027744},
MRREVIEWER = {John\ M.\ Mackay},
       DOI = {10.1016/j.jfa.2019.108298},
       URL = {https://doi-org.ezproxy.weizmann.ac.il/10.1016/j.jfa.2019.108298},
}

@article {ShulmanHadwin18b,
    AUTHOR = {Hadwin, Don and Shulman, Tatiana},
     TITLE = {Stability of group relations under small {H}ilbert-{S}chmidt
              perturbations},
   JOURNAL = {J. Funct. Anal.},
  FJOURNAL = {Journal of Functional Analysis},
    VOLUME = {275},
      YEAR = {2018},
    NUMBER = {4},
     PAGES = {761--792},
      ISSN = {0022-1236,1096-0783},
   MRCLASS = {20F05 (46L10)},
  MRNUMBER = {3807776},
MRREVIEWER = {Zhuang\ Niu},
       DOI = {10.1016/j.jfa.2018.05.006},
       URL = {https://doi-org.ezproxy.weizmann.ac.il/10.1016/j.jfa.2018.05.006},
}

@article{FournierFacioWillett26,
      AUTHOR = {Fournier-Facio, Francesco and Willett, Rufus},
      TITLE  = {The Local Lifting Property, Property FD, and stability of approximate representations},
      YEAR   = {2026},
      JOURNAL = {arXiv:2603.18456},
      URL ={https://arxiv.org/abs/2603.18456v2},
}

@article{DorOnHallKachkovskiy25,
      AUTHOR = {Dor-On, Adam and Hall, Lucas and Kachkovskiy, Ilya},
      TITLE  = {On almost commuting unitary matrices},
      YEAR   = {2025},
      JOURNAL = {arXiv:2510.03674},
      URL ={https://arxiv.org/abs/2510.03674},
}

@article{GerholdShalit24,
      AUTHOR = {Gerhold, Malte and Shalit,  Orr},
      TITLE  = {Dilation distance and the stability of ergodic commutation relations},
      YEAR   = {2024},
      JOURNAL = {arXiv:2406.05864},
      URL ={https://arxiv.org/abs/2406.05864},
}

@article{Eckhardt25,
      AUTHOR = {Eckhardt, Caleb},
      TITLE  = {Residually finite amenable groups that are not {H}ilbert-{S}chmidt stable},
      YEAR   = {2025},
      JOURNAL = {arXiv:2501.07791},
      URL ={https://arxiv.org/pdf/2501.07791},
}

@article {EckhardtShulman23,
    AUTHOR = {Eckhardt, Caleb and Shulman, Tatiana},
     TITLE = {On amenable {H}ilbert-{S}chmidt stable groups},
   JOURNAL = {J. Funct. Anal.},
  FJOURNAL = {Journal of Functional Analysis},
    VOLUME = {285},
      YEAR = {2023},
    NUMBER = {3},
     PAGES = {Paper No. 109954, 31},
      ISSN = {0022-1236,1096-0783},
   MRCLASS = {20F18 (20E22 46L35)},
  MRNUMBER = {4579918},
MRREVIEWER = {Paul\ Jolissaint},
       DOI = {10.1016/j.jfa.2023.109954},
       URL = {https://doi-org.ezproxy.weizmann.ac.il/10.1016/j.jfa.2023.109954},
}

@article {GongLin98,
    AUTHOR = {Gong, Guihua and Lin, Huaxin},
     TITLE = {Almost multiplicative morphisms and almost commuting matrices},
   JOURNAL = {J. Operator Theory},
  FJOURNAL = {Journal of Operator Theory},
    VOLUME = {40},
      YEAR = {1998},
    NUMBER = {2},
     PAGES = {217--275},
      ISSN = {0379-4024,1841-7744},
   MRCLASS = {46L05 (46L35 46L80)},
  MRNUMBER = {1660385},
MRREVIEWER = {E.\ St\o rmer},
}

@article {EilLorPed99,
    AUTHOR = {Eilers, S\o ren and Loring, Terry A. and Pedersen, Gert K.},
     TITLE = {Morphisms of extensions of {$C^*$}-algebras: pushing forward
              the {B}usby invariant},
   JOURNAL = {Adv. Math.},
  FJOURNAL = {Advances in Mathematics},
    VOLUME = {147},
      YEAR = {1999},
    NUMBER = {1},
     PAGES = {74--109},
      ISSN = {0001-8708,1090-2082},
   MRCLASS = {46L05 (46L80 46M15)},
  MRNUMBER = {1725815},
MRREVIEWER = {Vladimir\ Manuilov},
       DOI = {10.1006/aima.1999.1834},
       URL = {https://doi-org.ezproxy.weizmann.ac.il/10.1006/aima.1999.1834},
}

@article {Voiculescu83,
    AUTHOR = {Voiculescu, Dan},
     TITLE = {Asymptotically commuting finite rank unitary operators without
              commuting approximants},
   JOURNAL = {Acta Sci. Math. (Szeged)},
  FJOURNAL = {Acta Universitatis Szegediensis. Acta Scientiarum
              Mathematicarum},
    VOLUME = {45},
      YEAR = {1983},
    NUMBER = {1-4},
     PAGES = {429--431},
      ISSN = {0001-6969},
   MRCLASS = {47B44},
  MRNUMBER = {708811},
}

@article {ExelLoring89,
    AUTHOR = {Exel, Ruy and Loring, Terry},
     TITLE = {Almost commuting unitary matrices},
   JOURNAL = {Proc. Amer. Math. Soc.},
  FJOURNAL = {Proceedings of the American Mathematical Society},
    VOLUME = {106},
      YEAR = {1989},
    NUMBER = {4},
     PAGES = {913--915},
      ISSN = {0002-9939,1088-6826},
   MRCLASS = {15A15 (15A27 47A55 47B47 55M25)},
  MRNUMBER = {975641},
MRREVIEWER = {Kenneth\ R.\ Davidson},
       DOI = {10.2307/2047274},
       URL = {https://doi-org.ezproxy.weizmann.ac.il/10.2307/2047274},
}

@article {Exel93,
    AUTHOR = {Exel, Ruy},
     TITLE = {The soft torus and applications to almost commuting matrices},
   JOURNAL = {Pacific J. Math.},
  FJOURNAL = {Pacific Journal of Mathematics},
    VOLUME = {160},
      YEAR = {1993},
    NUMBER = {2},
     PAGES = {207--217},
      ISSN = {0030-8730,1945-5844},
   MRCLASS = {46L80 (46L05)},
  MRNUMBER = {1233352},
MRREVIEWER = {Carla\ Farsi},
       URL = {http://projecteuclid.org.ezproxy.weizmann.ac.il/euclid.pjm/1102624214},
}

@article {Question,
    AUTHOR = {Paulsen, Vern I. and Rahaman, Mizanur},
     TITLE = {Bisynchronous games and factorizable maps},
   JOURNAL = {Ann. Henri Poincar\'{e}},
  FJOURNAL = {Annales Henri Poincar\'{e}. A Journal of Theoretical and
              Mathematical Physics},
    VOLUME = {22},
      YEAR = {2021},
    NUMBER = {2},
     PAGES = {593--614},
      ISSN = {1424-0637,1424-0661},
   MRCLASS = {91A44 (05C57 81R50 91A43 91A81)},
  MRNUMBER = {4205237},
MRREVIEWER = {Ivan\ G.\ Todorov},
       DOI = {10.1007/s00023-020-01003-2},
       URL = {https://doi-org.ezproxy.weizmann.ac.il/10.1007/s00023-020-01003-2},
}
\end{document}